\documentclass[11pt,
final,
english
]{article}

\usepackage[utf8]{inputenc}    
\usepackage[T1]{fontenc}       
\usepackage{longtable}         
\usepackage{calc}			   %
\usepackage{exscale}           
\usepackage{array}             
\usepackage{subcaption} 
\usepackage{wrapfig} 
\usepackage{gitinfo2} 
\usepackage{xspace}            
\usepackage{stackengine} 	
\setstackEOL{\\}

\usepackage{faktor}			   
\usepackage{ifdraft}           
\usepackage{xparse} 
\usepackage{graphicx}   
\usepackage{appendix}
\usepackage{geometry}
\usepackage[multiuser]{fixme}  
\usepackage[expansion=false    
]{microtype}        
\usepackage{fancyhdr}          
\usepackage{xkeyval}           
\usepackage{enumitem}
\usepackage{tikz}              
\usepackage[english,strings]{babel}     
\usepackage[
backref = true,
style = alphabetic,
maxcitenames = 99,
maxbibnames = 99,
mincitenames = 5]{biblatex}
\usepackage[
final,                    
plainpages=false,              
pdfpagelabels=true,            
pdfencoding=auto,              
unicode=true,                  
hypertexnames=true,            
naturalnames=true,              
colorlinks,
linkcolor={blue!70!black},
citecolor={blue!50!black},
urlcolor={blue!80!black}
]{hyperref}                    
\usepackage{nchairx} 
\usepackage{csquotes}

\usetikzlibrary[patterns,patterns.meta,cd]

\graphicspath{{../tikz/pic/}{../tikz/plot/}}

\ifdraft{\synctex=1}{}

\FXRegisterAuthor{md}{anmd}{\color{blue}Marvin}

\newcommand{\myname}{\textbf{Marvin Dippell}}
\newcommand{\myemail}{\texttt{mdippell@unisa.it}}
\newcommand{\myaddress}{\begin{minipage}{8cm}
		\centering\small
		Dipartimento di Matematica\\
		Università degli Studi di Salerno\\
		via Giovanni Paolo II, 132\\
		84084 Fisciano (SA)\\
		Italy
	\end{minipage}
	\\}

\newcommand{\AuthorEmailOne}{\myemail}

\newcommand{\AuthorOne}{\myname}

\author{\AuthorOne\thanks{\AuthorEmailOne}{ }\\[0.5cm]
	       \myaddress\\
	}

\makeatletter

\newcommand{\bibnote}[2]{\nocite{#1}\@namedef{#1chairxnote}{#2}}

\renewcommand{\to}{%
  \relax\if@display
    \expandafter\longrightarrow
  \else
    \expandafter\rightarrow
  \fi
}
\makeatother

\newcommand{\Shuffle}{\mathop{Sh}}
\newcommand{\shuffle}{\mathrm{sh}}
\newcommand{\RedDelta}{\cc{\Delta}}

\newcommand{\RedDeltaSh}{{\RedDelta_\shuffle}}

\newcommand{\RedSym}{\cc{\Sym}}

\newcommand{\Reals}{\mathbb{R}}

\newcommand{\Total}{{\scriptscriptstyle\script{T}}}
\newcommand{\Wobs}{{\scriptscriptstyle\script{N}}}
\newcommand{\Null}{{\scriptscriptstyle\script{0}}}
\newcommand{\Van}{{\scriptscriptstyle\script{I}}}
\newcommand{\TotalNotWobs}{{\scriptscriptstyle\faktor{\Total}{\Wobs}}}
\newcommand{\TotalNotNull}{{\scriptscriptstyle\faktor{\Total}{\Null}}}
\newcommand{\WobsNotNull}{{\scriptscriptstyle\faktor{\Wobs}{\Null}}}
\newcommand{\NullNotVan}{{\scriptscriptstyle\faktor{\Null}{\Van}}}

\newcommand{\WOBS}{\script{N}}
\newcommand{\NULL}{\script{0}}

\newcommand{\strtensor}[1][]{\mathbin{\boxtimes_{\scriptscriptstyle{#1}}}}

\newcommand{\ConDiffop}{\Diffop}
\newcommand{\ConVectFields}{\VectFields}

\newcommand{\ConCinfty}{\Cinfty}

\newcommand{\vanishing}{\spacename{I}}

\newcommand{\VectFields}{\mathfrak{X}}

\newcommand{\HC}{\operator{C}}
\newcommand{\HCdiff}{\HC_{\script{diff}}}
\newcommand{\HH}{\operator{H}}
\newcommand{\HHdiff}{\HH_{\script{diff}}}
\newcommand{\Ca}{{\script{ca}}}
\newcommand{\CCa}{\operator{C}_\Ca}
\newcommand{\HCa}{\operator{H}_\Ca}
\newcommand{\hkr}{\operator{hkr}}%
\DeclareNameAlias{default}{family-given}
\title{Infinitesimal Star Products Compatible with Coisotropic Reduction}

\date{}

\begin{document}

\selectlanguage{english}

\maketitle	

\begin{abstract}
We determine infinitesimal star products on Poisson manifolds compatible with coisotropic reduction.
This is achieved by computing the second constraint Hochschild cohomology
of the constraint algebra of functions associated to any submanifold equipped with a simple distribution.
\end{abstract}

\tableofcontents
\newpage

\section{Introduction}
\label{sec:Introduction}

In deformation quantization \cite{bayen.et.al:1978a} a star product on a manifold $M$ is given by a deformation
in the sense of Gerstenhaber \cite{gerstenhaber:1963a,gerstenhaber:1964a,
gerstenhaber:1966a} of 
the point-wise product on $\Cinfty(M)$ by differential operators.
More specifically, it is given by an associative multiplication
$\star = \sum_{r=0}^{\infty} \hbar^r C_r$ on $\Cinfty(M)\formal{\hbar}$
such that $C_0$ is the point-wise product and all $C_r$ are bi-differential operators.
Given such a star product one can recover a Poisson structure on $M$ and the star product is considered to be a quantization of this Poisson manifold.
An important tool in deformation quantization is the differential Hochschild cohomology
$\HHdiff^\bullet(M)$, which controls the deformation theory of $\Cinfty(M)$.
In particular, equivalence classes of infinitesimal star products, i.e. deformations of the point-wise product up to first order, are given by the second Hochschild cohomology $\HHdiff^2(M)$.
Thus computing the Hochschild cohomology is of utmost importance in deformation quantization and
has been achieved by the classical Hochschild-Kostant-Rosenberg Theorem \cite{hochschild.kostant.rosenberg:1962a}, showing that
$\HHdiff^\bullet(M)$ agrees with the multivector fields on $M$.
The existence and classification of star products on Poisson manifolds was famously settled by Kontsevich \cite{kontsevich:1997:prea, kontsevich:2003a} by extending the Hochschild-Kostant-Rosenberg map to a quasi-isomorphism of $L_\infty$-algebras.

If a Poisson manifold $(M,\pi)$ carries the additional structure of a coisotropic submanifold
$C \subseteq M$ with a simple characteristic distribution $D \subseteq TC$ 
one can construct a Poisson structure $\pi_\red$ on the reduced manifold
$M_\red \coloneqq C/D$.
From the point of view of deformation quantization the question arises if one can find a star product on $M$ 
that is compatible with coisotropic reduction, i.e. which induces a star product on $M_\red$.
This is a special instance of the question if quantization commutes with reduction.
This problem has been studied in many different contexts over the years, see e.g. \cite{fedosov:1998a,reichert:2017a} for the case of Marsden-Weinstein reduction on symplectic manifolds,
\cite{bordemann.herbig.waldmann:2000a} for a BRST-like reduction, or \cite{esposito.kraft.schnitzer:2022a} for reduction with momentum maps on Poisson manifolds.

To solve this problem in the general case of a coisotropic submanifold of a Poisson manifold in \cite{dippell.esposito.waldmann:2019a, dippell.esposito.waldmann:2022a} \emph{constraint algebras} and their deformation theory were introduced.
The basic idea is to include the information needed for reduction, i.e. the submanifold $C$ and the distribution $D$, into a single object $\mathcal{M} = (M,C,D)$ called a \emph{constraint manifold}.
Then it can be shown that star products compatible with reduction correspond to deformations
of the constraint algebra $\Cinfty(\mathcal{M}) = (\Cinfty(M),\Cinfty(M)^D,\vanishing_C)$,
where $\Cinfty(M)^D$ denotes functions which are on the submanifold $C$ invariant along the distribution $D$ and $\vanishing_C$ is the vanishing ideal of $C$.
The deformation theory of $\Cinfty(\mathcal{M})$ is now controlled by what we call the constraint Hochschild cohomology $\HHdiff^\bullet(\mathcal{M})_\Wobs$.
While the zeroth and first constraint Hochschild cohomologies have been computed in \cite{dippell.esposito.waldmann:2022a}
the higher constraint Hochschild cohomologies are unknown.

Our main result is the computation of the second constraint Hochschild cohomology $\HHdiff^2(\mathcal{M})_\Wobs$ in \autoref{thm:SecondConCohom}.
This is achieved with the help of a constraint symbol calculus for multi-differential operators and a careful investigation of symbols corresponding to exact but not constraint exact cochains.
Moreover, we will show in \autoref{prop:InfinitesimalDeformations} that the terms in this cohomology which are not given
by bivector fields classify infinitesimal star products which are equivalent as star products, but not equivalent
when considered to be compatible with reduction.

\paragraph{Acknowledgements}
The author wants to thank Chiara Esposito and Stefan Waldmann for all the fruitful discussions,
as well as Antonio de Nicola for helpful remarks.
This work was supported by the
National Group for Algebraic and Geometric Structures, and their
Applications (GNSAGA – INdAM).

\newpage
\section{Constraint Manifolds and Differential Operators}
\label{sec:ConManifolds}

We introduce the basic objects, such as constraint manifolds, their vector fields
and differential operators.
Instead of presenting the full constraint language
we will restrict ourselves to the bare minimum needed for this work.
A more conceptual treatment of the contents of \autoref{sec:ConManifoldsVectFields} as well as proofs can be found in 
\cite{dippell.kern:2023a}.
For the symbol calculus of constraint differential operators as presented in \autoref{sec:ConDiffOps}
we follow the thesis
\cite{dippell:2023b}.
It should be noted that what we now call a constraint algebra differs slightly
from the coisotropic triples of algebras as considered in 
\cite{dippell.esposito.waldmann:2019a,dippell.esposito.waldmann:2022a}.

\subsection{Constraint Manifolds and their Vector Fields}
\label{sec:ConManifoldsVectFields}

A \emph{constraint manifold}
is a tuple 
\begin{equation}
	\mathcal{M} = (M,C,D)
\end{equation}
consisting of a smooth manifold $M$
together with a closed embedded
submanifold $\iota \colon C \hookrightarrow M$
and a simple distribution $D \subseteq TC$.
Since $D$ is simple we know that the leaf space
$\mathcal{M}_\red \coloneqq C / D$
carries a canonical smooth structure.

Associated to every constraint manifold there is its 
algebra of real- or complex-valued functions 
$\Cinfty(\mathcal{M}) \coloneqq
(\Cinfty(M), \Cinfty(\mathcal{M})_\Wobs, \Cinfty(\mathcal{M})_\Null)$
given by
\begin{equation}
\begin{split}
	\ConCinfty(\mathcal{M})_\Wobs
	&\coloneqq \left\{ f \in \Cinfty(M)
	\mid \Lie_X f\at{C} = 0 \text{ for all } X \in \Secinfty(D) 
	\right\},\\
	\ConCinfty(\mathcal{M})_\Null
	&\coloneqq \left\{ f \in \Cinfty(M)
	\mid f\at{C} = 0 \right\}.
\end{split}
\end{equation}
It forms a so-called \emph{strong constraint algebra},
in the sense that
$\ConCinfty(\mathcal{M})_\Wobs
\subseteq \ConCinfty(M)$
is a subalgebra and
$\ConCinfty(\mathcal{M})_\Null \subseteq \ConCinfty(\mathcal{M})_\Wobs$
is a two-sided ideal in $\ConCinfty(M)$.

As constraint manifolds correspond to strong constraint algebras,
so constraint vector bundles correspond to strong constraint modules.
For us the most important strong constraint module will be
the constraint vector fields 
$\VectFields(\mathcal{M})
\coloneqq (\VectFields(M),
\VectFields(\mathcal{M})_\Wobs,
\VectFields(\mathcal{M})_\Null)$
given by
\begin{equation}
\begin{split}
	\VectFields(\mathcal{M})_\Wobs
	&\coloneqq \left\{ X \in \Secinfty(TM) \bigm| X\at{C} \in \Secinfty(TC) \text{ and}\right.\\
	&\hspace{8.5em}\left.[X,Y] \in \Secinfty(D) \text{ for all } Y \in \Secinfty(D) \right\},\\
	\VectFields(\mathcal{M})_\Null
	&\coloneqq \left\{ X \in \Secinfty(TM) \bigm| X\at{C} \in \Secinfty(D) \right\}.
\end{split}
\end{equation}
This is a \emph{strong constraint module} over the strong constraint algebra
$\Cinfty(\mathcal{M})$ in the sense that
$\VectFields(\mathcal{M})_\Wobs \subseteq \VectFields(M)$
is a
$\Cinfty(\mathcal{M})_\Wobs$-submodule
and
$\VectFields(\mathcal{M})_\Null \subseteq \VectFields(\mathcal{M})_\Wobs$
is a 
$\Cinfty(M)$-submodule
of $\VectFields(M)$
such that additionally
$\Cinfty(\mathcal{M})_\Null \cdot \VectFields(M)
\subseteq \VectFields(\mathcal{M})_\Null$ holds.
From now on, when introducing new constraint objects, we will only specify the $\WOBS$- and $\NULL$-components.

Out of these constraint vector fields we can construct a graded strong constraint algebra
of multivector fields $\VectFields(\mathcal{M}) = \bigoplus_{k=0}^\infty \VectFields^k(\mathcal{M})$ by setting
\begin{equation}
\begin{split}
	\ConVectFields^k(\mathcal{M})_\Wobs
	&\coloneqq \Anti^k\VectFields(\mathcal{M})_\Wobs
	+ \Anti^{k-1}\VectFields(M)
	\wedge \VectFields(\mathcal{M})_\Null,
	\\
	\ConVectFields^k(\mathcal{M})_\Null
	&\coloneqq \Anti^{k-1} \VectFields(M)
	\wedge \VectFields(\mathcal{M})_\Null.
\end{split}
\end{equation}
Another related construction is that of the graded symmetric constraint algebra
$\Sym\VectFields(\mathcal{M}) = \bigoplus_{k=0}^\infty\Sym^k\VectFields(\mathcal{M})$
given by
\begin{equation}
\begin{split}
	(\Sym^k \VectFields(\mathcal{M}))_\Wobs
	&\coloneqq \Sym^k \VectFields(\mathcal{M})_\Wobs,\\
	(\Sym^k \VectFields(\mathcal{M}))_\Null
	&\coloneqq \Sym^{k-1} \VectFields(\mathcal{M})_\Wobs
	\vee \VectFields(\mathcal{M})_\Null.
\end{split}
\end{equation}
Note, that even though $\Sym \VectFields(\mathcal{M})$
is a strong constraint $\Cinfty(\mathcal{M})$-module
it is in general not a \emph{strong} constraint algebra, in the sense that
$(\Sym \VectFields(\mathcal{M}))_\Null$
is a two-sided ideal only in
$(\Sym \VectFields(\mathcal{M}))_\Wobs$.
We will also need the reduced symmetric algebra
$\RedSym\ConVectFields(\mathcal{M})$
defined by
$\RedSym\VectFields(\mathcal{M})
\coloneqq \bigoplus_{k=1}^\infty\Sym^k\VectFields(\mathcal{M})$.

\begin{remark}
	In the language of \cite{dippell.kern:2023a} we constructed
	$\Anti\ConVectFields(\mathcal{M})$ using the tensor product $\strtensor$,
	while in the construction of $\Sym\ConVectFields(\mathcal{M})$
	the tensor product $\tensor$ was used.
\end{remark}

All of the above constructions are compatible with reduction in the sense that the
quotient of their $\WOBS$- by their $\NULL$-components yields the classical objects
on $\mathcal{M}_\red$.
In particular, we have
\begin{align}
	\Cinfty(\mathcal{M})_\red 
	&\coloneqq \Cinfty(\mathcal{M})_\Wobs / \Cinfty(\mathcal{M})_\Null
	\simeq \Cinfty(\mathcal{M}_\red),
	\\
	\ConVectFields^k(\mathcal{M})_\red
	&\coloneqq \ConVectFields^k(\mathcal{M})_\Wobs / \ConVectFields^k(\mathcal{M})_\Null
	\simeq \ConVectFields^k(\mathcal{M}_\red),
	\\
	\Sym^k\VectFields(\mathcal{M})_\red
	&\coloneqq (\Sym^k\VectFields(\mathcal{M}))_\Wobs / (\Sym^k\VectFields(\mathcal{M}))_\Null
	\simeq \Sym^k\VectFields(\mathcal{M}_\red).
\end{align}

\begin{example}\
	\begin{examplelist}
		\item Let $(M,\pi)$ be a Poisson manifold and $C \subseteq M$ a coisotropic submanifold.
		Moreover, assume that the characteristic distribution $D_C \subseteq TC$ spanned by the hamiltonian vector fields $X_f$ for $f \in \vanishing_C$
		is simple.
		Then $\mathcal{M} = (M,C,D_C)$ is a constraint manifold and $\mathcal{M}_\red = M / D_C$ is the manifold obtained by the usual coisotropic reduction procedure.
		Moreover, we have $\pi \in \ConVectFields^2(\mathcal{M})_\Wobs$
		and $[\pi] \in \ConVectFields^2(\mathcal{M})_\red$
		is the reduced Poisson bivector field on $\mathcal{M}_\red$.
		\item Conversely, given a constraint manifold $\mathcal{M} = (M,C,D)$
		and a Poisson bivector field $\pi \in \ConVectFields^2(\mathcal{M})_\Wobs$
		it is not hard to see that $C \subseteq (M,\pi)$ is a coisotropic submanifold with characteristic distribution $D_C \subseteq D$.
	\end{examplelist}	
\end{example}

\subsection{Constraint Differential Operators and their Symbols}
\label{sec:ConDiffOps}

For a smooth manifold $M$ we denote the multi-differential operators of $\Cinfty(M)$ by
$\Diffop^n(M)$,
where $n$ is the number of inputs.
We will always assume that (multi-)differential operators vanish on constant functions.
Given a torsion-free covariant derivative $\nabla$ on $M$
it is well-known that one can construct a full symbol calculus for $\Diffop^\bullet(M)$, in the sense that one obtains an isomorphism
\begin{equation}
	\Op^\nabla
		\colon
	\Tensor^\bullet \RedSym \VectFields(M)
		\to
	\Diffop^\bullet(M)
\end{equation}
of filtered $\Cinfty(M)$-modules,
see \cite[Chap.~IV,
§9]{palais:1965a} or \cite[App.~A]{waldmann:2007a} for a modern textbook.
Here $\Tensor^\bullet\RedSym\VectFields(M)$ denotes the tensor algebra of the 
reduced symmetric algebra of $\VectFields(M)$ graded by tensor factors.

Now given a constraint manifold $\mathcal{M} = (M,C,D)$ one can equip the multi-differential operators with a constraint structure.
For this we define
$\ConDiffop^n(\mathcal{M})_\Wobs$
to be those multi-differential operators 
$D \in \Diffop^n(M)$
such that
\begin{align}
	D(f_1, \dotsc, f_n) \in \Cinfty(\mathcal{M})_\Wobs
	&\text{ if }
	f_1,\dotsc, f_n \in \Cinfty(\mathcal{M})_\Wobs
	\\
	\shortintertext{and}
	\begin{split}
		D(f_1, \dotsc, f_n) \in \Cinfty(\mathcal{M})_\Null
		&\text{ if }
		f_1,\dotsc, f_n \in \Cinfty(\mathcal{M})_\Wobs
		\\
		&\text{ and }
		f_i \in \Cinfty(\mathcal{M}_\Null)
		\text{ for some }
		i \in \{1,\dotsc,n\}.
	\end{split}
\intertext{Moreover, we define $\ConDiffop^n(\mathcal{M})_\Null$
as those multi-differential operators
$D \in \Diffop^n(M)$
such that}
D(f_1, \dotsc, f_n) \in \Cinfty(\mathcal{M})_\Null
&\text{ if }
f_1,\dotsc, f_n \in \Cinfty(\mathcal{M})_\Wobs.
\end{align}
In order to extent the symbol calculus to constraint manifolds we need to consider \emph{constraint covariant derivatives},
which are covariant derivatives on $M$ such that additionally
we have
\begin{equation}
\begin{split}
	\nabla_X Y \in \VectFields(\mathcal{M})_\Wobs
	&\text{ for all }
	X \in \VectFields(\mathcal{M})_\Wobs
	\text{ and }
	Y \in \VectFields(\mathcal{M})_\Wobs,
	\\
	\nabla_X Y \in \VectFields(\mathcal{M})_\Null
	&\text{ for all }
	X \in \VectFields(\mathcal{M})_\Wobs
	\text{ and }
	Y \in \VectFields(\mathcal{M})_\Null,
	\\
	\nabla_X Y \in \VectFields(\mathcal{M})_\Null
	&\text{ for all }
	X \in \VectFields(\mathcal{M})_\Null
	\text{ and }
	Y \in \VectFields(\mathcal{M})_\Wobs.
\end{split}
\end{equation}
As a last ingredient for the constraint symbol calculus we need to consider a constraint version of $\Tensor^\bullet\RedSym\VectFields(M)$.
Thus let us define
\begin{equation}
\begin{split}
	\bigl(\Tensor^n\RedSym\VectFields(\mathcal{M})\bigr)_\Wobs
	&\coloneqq \Tensor^n (\RedSym\VectFields(\mathcal{M}))_\Wobs
	+ \sum_{i=1}^n \Tensor^{i-1}\Sym \VectFields(M)
	\tensor (\Sym\VectFields(\mathcal{M}))_\Null
	\tensor \Tensor^{n-i}\Sym\VectFields(M),
	\\
	\bigl(\Tensor^n\RedSym\VectFields(\mathcal{M})\bigr)_\Null
	&\coloneqq \sum_{i=1}^n \Tensor^{i-1}\Sym \VectFields(M)
	\tensor (\Sym\VectFields(\mathcal{M}))_\Null
	\tensor \Tensor^{n-i}\Sym\VectFields(M).
\end{split}
\end{equation}
With this we obtain a constraint version of the symbol calculus.

\begin{proposition}[Constraint multisymbol calculus]
	\label{prop:ConMultSymbolCalculus}
	Let $\mathcal{M} = (M,C,D)$ be a constraint manifold.
	\begin{propositionlist}
		\item There exists a torsion-free constraint covariant derivative on $\mathcal{M}$.
		\item For any torsion-free constraint covariant derivative $\nabla$ the associated isomorphism
		\begin{equation}
			\Op^\nabla
				\colon
			\Tensor^\bullet \RedSym \VectFields(M)
				\to
			\ConDiffop^\bullet(M),
		\end{equation}
		restricts to isomorphisms
		\begin{align}
			\label{eq:RestrictedOp_N}
			\Op^\nabla
				\colon
			\bigl(\Tensor^\bullet\RedSym\VectFields(\mathcal{M})\bigr)_\Wobs
				\to
			\ConDiffop^\bullet(\mathcal{M})_\Wobs,
			\\
			\shortintertext{and}
			\Op^\nabla
				\colon
			\bigl(\Tensor^\bullet \RedSym \VectFields(\mathcal{M})\bigr)_\Null
				\to
			\ConDiffop^\bullet(\mathcal{M})_\Null.
			\label{eq:RestrictedOp_0}
		\end{align}
	\end{propositionlist}	
\end{proposition}

\begin{proof}
	By \cite[Thm. 3.32]{dippell.kern:2023a} there exists a dual basis of $\VectFields(M)$ which is adapted to the submanifold $C$
	and the distribution $D$ such that $Y = \sum_{i=1}^N e^i(Y) e_i$ for any $Y \in \VectFields(M)$.
	Then it is an easy check that
	\begin{equation*}
		\nabla_X Y \coloneqq \sum_{i=1}^{N} \Lie_X(e^i(Y))e_i,
	\end{equation*}
	where $X,Y \in \VectFields(M)$,
	defines a constraint covariant derivative on $\mathcal{M}$.
	For the second part, note that one way to define $\Op^\nabla$ is by
	\begin{equation*}
		\label{eq:OpX}
            \Op^\nabla(X_1 \tensor \dots \tensor X_n)(f_1, \ldots, f_n)
            =
            \frac{1}{k_1! \cdots k_n!}
            \inss(X) \big(
            \SymD^{k_1} f_1 \tensor \cdots \tensor \SymD^{k_n} f_n
            \big),
	\end{equation*}
	where $X_i \in \RedSym^{k_i}\VectFields(M)$ and $f_i \in \Cinfty(M)$ for all $i = 1,\dotsc, n$.
	Here
	$D \colon \Sym^k\Secinfty(T^*M)
	\to \Sym^{k+1}\Secinfty(T^*M)$
	is the symmetrized covariant derivative
	defined by
	\begin{equation*}
		(\SymD \alpha)(X_0, \ldots, X_k)
		=
		\sum_{i=0}^k
		\nabla_{X_i} \alpha(X_0, \stackrel{i}{\ldots}, X_k)
		-
		\sum_{i \ne j}
		\alpha(\nabla_{X_i} X_j,
		X_0, \stackrel{i}{\ldots} \, \stackrel{j}{\ldots}, X_k),
	\end{equation*}
	where $X_0, \dotsc, X_k \in \VectFields(M)$.
	With these formulas it is now straightforward to check that $\Op^\nabla$
	restricts as in \eqref{eq:RestrictedOp_N} and \eqref{eq:RestrictedOp_0}.
\end{proof}

Note again that all the above results are compatible with reduction.
More precisely, every constraint covariant derivative induces
a covariant derivative $\nabla^\red$ on $\mathcal{M}_\red$
and its associated symbol calculus $\Op^{\nabla^\red}$
agrees with the isomorphism induced by $\Op^{\nabla}$.

\section{Constraint Hochschild Cohomology and Infinitesimal Deformations}
\label{sec:ConStarProducts}

In deformation quantization it is well-known that the quantization of Poisson structures
by differential operators is governed by the differential Hochschild complex
$\HCdiff^\bullet(M) \coloneqq \Diffop^\bullet(M)$ of the algebra $\Cinfty(M)$.
We recall some basics of the deformation theory of constraint algebras from \cite{dippell.esposito.waldmann:2022a} before explicitly computing the second constraint Hochschild cohomology.

\subsection{Differential Constraint Hochschild Cohomology}
\label{sec:ConHochschildCohomology}

In \cite{dippell.esposito.waldmann:2022a} it was shown that deformations of algebras compatible with reduction
are governed by the constraint Hochschild complex.
Since we are interested in deformations by differential operators we will in the following consider the constraint differential Hochschild complex
$\HCdiff^\bullet(\mathcal{M})$, which is explicitly given by
\begin{equation}
\begin{split}
	\HCdiff^n(\mathcal{M})_\Wobs
	&\coloneqq \Diffop^n(\mathcal{M})_\Wobs,
	\\
	\HCdiff^n(\mathcal{M})_\Null
	&\coloneqq \Diffop^n(\mathcal{M})_\Null,
\end{split}
\end{equation}
together with the classical Hochschild differential
$\delta \colon \HCdiff^\bullet(M) \to \HCdiff^{\bullet+1}(M)$
given by
\begin{equation}
\begin{split}
	(\delta D)(f_0,\dotsc, f_n)
		&= f_0 D(f_1,\dotsc, f_n)
		+ (-1)^n D(f_0,\dotsc,f_{n-1}) f_n\\
		&\quad+ \sum_{i=0}^{n} (-1)^{i+1} D(f_0,\dotsc, f_i f_{i+1},\dotsc, f_n),
\end{split}
\end{equation}
for $D \in \HCdiff^n(M)$, and $f_0,\dotsc,f_n \in \Cinfty(M)$.
It is easy to see that $\HCdiff^n(\mathcal{M})_\Wobs$
is a subcomplex of $\HCdiff^n(M)$.
This allows us to consider the constraint Hochschild cohomology
$\HHdiff^\bullet(\mathcal{M})_\Wobs$ 
given by the cohomology of the subcomplex
$\HCdiff^\bullet(\mathcal{M})_\Wobs$,
as well as $\HHdiff^\bullet(\mathcal{M})_\Null \subseteq \HHdiff^\bullet(\mathcal{M})_\Wobs$ induced by the subset $\HCdiff^\bullet(\mathcal{M})_\Null \subseteq \HCdiff^\bullet(\mathcal{M})_\Wobs$.
Moreover, there are canonical morphisms $\HHdiff^\bullet(\mathcal{M})_\Wobs \to \HHdiff^\bullet(M)$ and
$\HHdiff^\bullet(\mathcal{M})_\Wobs \to \HHdiff^\bullet(\mathcal{M}_\red)$.
For more details, we refer to \cite{dippell.esposito.waldmann:2022a}.

The classical Hochschild-Kostant-Rosenberg Theorem computes the differential Hochschild
cohomology for a given smooth manifold $M$, see \cite{hochschild.kostant.rosenberg:1962a} as well as \cite{dippell.esposito.schnitzer.waldmann:2024a:pre} for a recent improved version.
In particular, it states that 
$\hkr \colon \VectFields^\bullet(M) \to \HCdiff^\bullet(M)$
given by
\begin{equation}
	\label{eq:hkrMap}
	\hkr(X_1 \wedge \dots \wedge X_n)
	\coloneqq
	\frac{1}{n!} \sum_{\sigma \in S_n} \sign(\sigma) \cdot \Lie_{X_{\sigma(1)}} \tensor \dots \tensor \Lie_{X_{\sigma(n)}}
\end{equation}
is a quasi-isomorphism, i.e. an isomorphism in cohomology.
Here we consider $\VectFields^\bullet(M)$ to be equipped with the zero differential.

We now return to the constraint situation.
\begin{proposition}
	\label{prop:hkrToWobs}
	Let $\mathcal{M} = (M,C,D)$ be a constraint manifold.
	\begin{propositionlist}
		\item The map
		$\hkr \colon
		\VectFields^\bullet(M)
		\to
		\HCdiff^\bullet(M)$
		restricts to a morphism
		$\hkr \colon \ConVectFields^\bullet(\mathcal{M})_\Wobs \to \HCdiff^\bullet(\mathcal{M})_\Wobs$
		of complexes.
		\item In degree $0$ and $1$ the morphism $\hkr$ is an isomorphism in constraint cohomology, in particular we have
		\begin{equation}
			\HHdiff^0(\mathcal{M})_\Wobs \simeq \Cinfty(\mathcal{M})_\Wobs
			\qquad\text{and}\qquad
			\HHdiff^1(\mathcal{M})_\Wobs \simeq \VectFields(\mathcal{M})_\Wobs.
		\end{equation}
		\item The morphism in constraint cohomology induced by $\hkr$ is injective in all degrees.
	\end{propositionlist}	
\end{proposition}

\begin{proof}
	The first part directly follows from the definitions of 
	$\ConVectFields^\bullet(\mathcal{M})_\Wobs$
	and $\HCdiff^\bullet(\mathcal{M})_\Wobs$,
	while the second part was shown for general constraint algebras
	in \cite{dippell.esposito.waldmann:2022a}.
	For the last part note that
	$\hkr(X) = [0] \in \HHdiff^\bullet(\mathcal{M})_\Wobs$
	if and only $\hkr(X) = \delta D$
	for some $D \in \HCdiff^{\bullet-1}(\mathcal{M})_\Wobs \subseteq \HCdiff^{\bullet-1}(M)$.
	Now since $\hkr(X)$ is totally antisymmetric and $\delta D$ is not, it follows $\hkr(X) = 0$ and thus by injectivity of the classical HKR map 
	$X = 0$.
\end{proof}

Nevertheless, $\hkr$ is not an isomorphism in constraint cohomology, as the next example shows.

\begin{example}
	Consider the constraint manifold
	$\mathcal{M} = (\Reals^{n_\Total}, \Reals^{n_\Wobs}, \Reals^{n_\Null})$
	with $n_\Total > n_\Wobs > n_\Null > 0$,
	where we consider $\Reals^{n_\Null}$ as a distribution on the submanifold
	$\Reals^{n_\Wobs} \subseteq \Reals^{n_\Total}$.
	It is clear that
	\begin{equation}
		\frac{\del^2}{\del x^{n_\Null}\del x^{n_\Total}} \in \Diffop^1(\Reals^{n_\Total}) = \HCdiff^1(\Reals^{n_\Total})
	\end{equation}
	is not constraint, since
	for $f \coloneqq x^{n_\Null} x^{n_\Total} \in \Cinfty(\mathcal{M})_\Null$
	we have
	$\frac{\del^2f}{\del x^{n_\Null} \del x^{n_\Total}} = 1 \notin \Cinfty(\mathcal{M})_\Null$.
	Nevertheless, applying the Hochschild differential yields
	\begin{equation}
		\label{eq:ExNonTrivClass}
		\delta\left(\frac{\del^2}{\del x^{n_\Null} \del x^{n_\Total}}\right)
		= - \frac{\del}{\del x^{n_\Null}} \tensor \frac{\del}{\del x^{n_\Total}} - \frac{\del}{\del x^{n_\Total}} \tensor \frac{\del}{\del x^{n_\Null}},
	\end{equation}
	and a straightforward check shows that this is an element in 
	$\HCdiff^\bullet(\mathcal{M})_\Wobs$.
	Moreover, it is clear that there cannot exist a constraint potential for \eqref{eq:ExNonTrivClass}.
	Thus \eqref{eq:ExNonTrivClass} defines a non-trivial cohomology class in $\HCdiff^2(\mathcal{M})_\Wobs$.
	Finally, since it is symmetric it can not be in the image of $\hkr$.
\end{example}

\subsection{Second Constraint Hochschild Cohomology}

While the zeroth and first constraint Hochschild cohomologies have been computed in \autoref{prop:hkrToWobs} we will now focus on determining the missing terms in the second constraint Hochschild cohomology.
For this we first use the symbol calculus from \autoref{prop:ConMultSymbolCalculus} to identify the constraint differential operators with their symbols.
Thus from now on let $\nabla$ be a torsion-free constraint covariant derivative on a constraint manifold
$\mathcal{M} = (M,C,D)$.

It is well-known that
$\Op^\nabla
	\colon
\Tensor^\bullet\RedSym\VectFields(M)
	\to
\HCdiff^\bullet(M)$
becomes an isomorphism of complexes when we equip
$\Tensor^\bullet\RedSym\VectFields(M)$
with the differential $\D$ defined by
\begin{equation}
	\D(\phi_1 \tensor \dots \tensor \phi_n)
	\coloneqq
	\sum_{i=1}^{n} (-1)^{i} \phi_1 \tensor \cdots \tensor \RedDeltaSh(\phi_i) \tensor \cdots \tensor \phi_n,
\end{equation}
for $\phi_1, \cdots, \phi_n \in \RedSym\VectFields(M)$,
where $\RedDeltaSh$ denotes the reduced shuffle coproduct on $\RedSym \VectFields(M)$, i.e.
\begin{equation}
	\RedDeltaSh(X_1 \vee \dots \vee X_k)
	\coloneqq \sum_{\ell=1}^{k-1} \sum_{\sigma \in \Shuffle(\ell,k-\ell)}
	(X_{\sigma(1)} \vee \cdots \vee X_{\sigma(\ell)}) \tensor (X_{\sigma(\ell+1)}  \vee \cdots \vee X_{\sigma(k)})
\end{equation}
for $X_1,\dotsc,X_k \in \VectFields(M)$,
where we sum over all so-called \emph{$(\ell,k-\ell)$-shuffles}, i.e. permutations $\sigma$ with
$\sigma(1) < \dots < \sigma(\ell,)$ and $\sigma(\ell+1) < \dots < \sigma(k)$.
We will denote $\Tensor^\bullet \RedSym \VectFields(M)$ equipped with the differential
$\D$ by $\CCa^\bullet(M)$.

\newpage

\begin{lemma}
	\label{lem:KernelD}
	Let $M$ be a manifold.
	\begin{lemmalist}
		\item The differential $\D$ is injective on $\RedSym^k\VectFields(M)$ for all $k \geq 2$.
		\item The differential $\D$ vanishes on $\VectFields(M)$.
	\end{lemmalist}
\end{lemma}

\begin{proof}
	For the first part, note that
	\begin{equation*}
		\vee \circ \RedDeltaSh (X_1 \vee \dots \vee X_k) = (2^k - 2) \cdot (X_1 \vee \dots \vee X_k)
	\end{equation*}
	and hence $\RedDeltaSh$ is injective for $k \geq 2$.
	The second part follows directly from the definition of $\RedDeltaSh$.
\end{proof}

Identifying $\HCdiff^\bullet(M)$ with $\CCa^\bullet(M)$
the map \eqref{eq:hkrMap} becomes a quasi-isomorphism
$\hkr \colon \VectFields^\bullet(M) \to \CCa^\bullet(M)$ given by
\begin{equation}
	\hkr(X_1 \wedge \dots \wedge X_n)
	\coloneqq
	\AntiSymmetrizer(X_1 \tensor \dots \tensor X_n)
	=
	\frac{1}{n!} \sum_{\sigma \in S_n} \sign(\sigma) \cdot X_{\sigma(1)} \tensor \dots \tensor X_{\sigma(n)},
\end{equation}
for $X_1, \dotsc, X_n \in \VectFields(M)$.
In fact, it can be shown, see \cite{dippell.esposito.schnitzer.waldmann:2024a:pre}, that every
$\phi_1 \tensor \dots \phi_n \in \CCa^n(M)$ can be written as
\begin{equation}
	\label{eq:HKRRetract}
	\phi_1 \tensor \dots \tensor \phi_n
	= \hkr\bigl(\pr_1(\phi_1) \wedge \dots \wedge \pr_1(\phi_n)\bigr)
	+ \D H(\phi),
\end{equation}
where $\pr_1 \colon \RedSym\VectFields(M) \to \VectFields(M)$ denotes the projection onto
symmetric degree $1$ and $H(\phi) \in \CCa^{n-1}(M)$.

For a constraint manifold $\mathcal{M} = (M,C,D)$
it follows directly from \autoref{prop:ConMultSymbolCalculus}
that the map
$\Op^\nabla
	\colon
\bigl(\Tensor^\bullet\RedSym\ConVectFields(\mathcal{M})\bigr)_\Wobs
	\to
\ConDiffop^\bullet(\mathcal{M})_\Wobs$
is an isomorphism of complexes.
Hence we will in the following consider the complex
\begin{equation}
	\CCa^\bullet(\mathcal{M})_\Wobs
	\coloneqq (\Tensor^\bullet\RedSym\ConVectFields(\mathcal{M}))_\Wobs
\end{equation}
equipped with the differential $\D$.
In \autoref{prop:hkrToWobs} we showed that $\hkr$ induces an injection
\begin{equation}
	\hkr
		\colon
	\ConVectFields^2(\mathcal{M})_\Wobs
		\to
	\HCa^2(\mathcal{M})_\Wobs.
\end{equation}
In order to compute the full cohomology $\HCa^2(\mathcal{M})_\Wobs$
we need to find a complement of
$(\ConVectFields^2(\mathcal{M}))_\Wobs$
inside
$\HCa^2(\mathcal{M})_\Wobs$.
Since $\CCa^\bullet(\mathcal{M})_\Wobs$
is a subcomplex of $\CCa^\bullet(M)$
we need to understand which cocycles in $\CCa^\bullet(\mathcal{M})_\Wobs$ are exact, but do not admit a potential inside $\CCa^\bullet(\mathcal{M})_\Wobs$.
In other words:
Which non-constraint cochains
$\phi \in \CCa^\bullet(M)$
yield constraint coboundaries
$\D\phi \in \CCa^\bullet(\mathcal{M})_\Wobs$.

To explicitly compute the second constraint Hochschild cohomology we
need to choose some additional structures:
Let $U \subseteq M$ be a tubular neighbourhood of $C$ and choose a bump function $\chi$ with $\supp\chi \subseteq U$ and $\chi\at{C} = 1$.
With this we can define a prolongation map 
$\prol \colon \Cinfty(C) \to \Cinfty(M)$
such that $\iota^*\prol = \id_{\Cinfty(C)}$,
which induces an isomorphism
\begin{equation}
	\label{eq:DecompFunctions}
	\Cinfty(M) \simeq \vanishing_C \oplus \Cinfty(C)
\end{equation}
of vector spaces.

Similarly, we obtain a prolongation map 
$\prol \colon \Secinfty(\iota^\#TM) \to \Secinfty(TM)$ with
$\iota^\# \prol = \id_{\Secinfty(\iota^\#TM)}$ inducing an isomorphism
\begin{equation}
	\label{eq.DecompVectorFields}
	\VectFields(M)
		\simeq
	\vanishing_C \cdot \Secinfty(TM) \oplus \Secinfty(\iota^\#TM)
\end{equation}
of vector spaces.
Here $\iota^\#$ denotes the pull-back of vector bundles along the inclusion $\iota \colon C \hookrightarrow M$.
Note that $\Secinfty(\iota^\#TM)$ is given by sections of a vector bundle over the submanifold $C$.
We will in the following suppress this isomorphism in our notation and assume that any section on $C$ is extended to $M$ by $\prol$ if necessary.

Finally, by choosing subbundles
$D^\perp \subseteq TC$
and
$TC^\perp \subseteq \iota^\#TM$
such that
$TC = D \oplus D^\perp$
and
$\iota^\#TM = TC \oplus TC^\perp$
we obtain
\begin{equation}
	\label{eq:DecompPullBackBundle}
	\Secinfty(\iota^\#TM) 
		=
	\Secinfty(D) \oplus \Secinfty(D^\perp) \oplus \Secinfty(TC^\perp).
\end{equation}
Using these complements we introduce the following notations:
\begin{equation}
	\label{eq:SplittingSym}
\begin{split}
	\bigl(\RedSym^k\ConVectFields(\mathcal{M})\bigr)_\NullNotVan
	&\coloneqq
		\Sym^{k-1}\Secinfty(TC) \vee \Secinfty(D),
	\\
	\bigl(\RedSym^k\ConVectFields(\mathcal{M})\bigr)_\WobsNotNull
	&\coloneqq
	\RedSym^k \Secinfty(D^\perp),
	\\
	\bigl(\RedSym^k\ConVectFields(\mathcal{M})\bigr)_\TotalNotWobs
	&\coloneqq
	\Sym^{k-1}\Secinfty(\iota^\#TM) \vee \Secinfty(TC^\perp),
	\\
	\bigl(\RedSym^k\ConVectFields(\mathcal{M})\bigr)_\TotalNotNull
	&\coloneqq
	\Sym^{k-1}\Secinfty(\iota^\#TM) \vee \Secinfty(TC^\perp)
	\oplus
	\RedSym^k\Secinfty(D^\perp),
\end{split}
\end{equation}
for every $k\geq 1$.
Moreover, we will need the following two subspaces of $\Tensor^2\RedSym\VectFields(M)$:
\begin{equation}
	\label{eq:SplittingTensor}
\begin{split}
	\bigl(\Tensor^2\RedSym\ConVectFields(\mathcal{M})\bigr)_\NullNotVan
	&\coloneqq
	\bigl(\RedSym\ConVectFields(\mathcal{M})\bigr)_\NullNotVan
	\tensor
	\bigl(\RedSym\ConVectFields(\mathcal{M})\bigr)_\TotalNotNull
	\\
	&\phantom{\coloneqq}\oplus
	\bigl(\RedSym\ConVectFields(\mathcal{M})\bigr)_\TotalNotNull
	\tensor
	\bigl(\RedSym\ConVectFields(\mathcal{M})\bigr)_\NullNotVan
	\\
	&\phantom{\coloneqq}\oplus
	\bigl(\RedSym\ConVectFields(\mathcal{M})\bigr)_\NullNotVan
	\tensor
	\bigl(\RedSym\ConVectFields(\mathcal{M})\bigr)_\NullNotVan,
	\\
	\bigl(\Tensor^2\RedSym\ConVectFields(\mathcal{M})\bigr)_\TotalNotWobs
	&\coloneqq
	\bigl(\RedSym\ConVectFields(\mathcal{M})\bigr)_\TotalNotWobs
	\tensor
	\bigl(\RedSym\ConVectFields(\mathcal{M})\bigr)_\WobsNotNull
	\\
	&\phantom{\coloneqq}\oplus
	\bigl(\RedSym\ConVectFields(\mathcal{M})\bigr)_\WobsNotNull
	\tensor
	\bigl(\RedSym\ConVectFields(\mathcal{M})\bigr)_\TotalNotWobs
	\\
	&\phantom{\coloneqq}\oplus
	\bigl(\RedSym\ConVectFields(\mathcal{M})\bigr)_\TotalNotWobs
	\tensor
	\bigl(\RedSym\ConVectFields(\mathcal{M})\bigr)_\TotalNotWobs.
\end{split}
\end{equation}

\begin{lemma}
	\label{lem:DecompSymTenAlgebra}
	Let $\mathcal{M} = (M,C,D)$ be a constraint manifold.
		\begin{lemmalist}
		\item \label{prop:DecompSymTenAlgebra_Sym}
		There exists an isomorphism of vector spaces such that
		\begin{equation}
		\begin{split}
			\RedSym^k\VectFields(M)
			&\simeq \Cinfty(M) \cdot \bigl(\RedSym^k\ConVectFields(\mathcal{M})\bigr)_\Wobs
			\oplus \bigl(\RedSym^k\ConVectFields(\mathcal{M})\bigr)_\TotalNotWobs,
		\end{split}
		\end{equation}
		\item \label{prop:DecompSymTenAlgebra_Ten}
		There exists an isomorphism of vector spaces such that
		\begin{align}
			\Tensor^2\RedSym\ConVectFields(M)
			&\simeq \Cinfty(M) \cdot \bigl(\Tensor^2\RedSym\ConVectFields(\mathcal{M})\bigr)_\Wobs
			\oplus \bigl(\Tensor^2\RedSym\ConVectFields(\mathcal{M})\bigr)_\TotalNotWobs,
			\intertext{and}
			\bigl(\Tensor^2\RedSym\ConVectFields(\mathcal{M})\bigr)_\Null
			&\simeq \vanishing_C \cdot \Tensor^2\RedSym \VectFields(M)
			\oplus \bigl(\Tensor^2\RedSym\ConVectFields(\mathcal{M})\bigr)_\NullNotVan.
		\end{align}
	\end{lemmalist}
\end{lemma}

\begin{proof}
	Choose a tubular neighbourhood and a cut-off function for $C$ as well as complements of $D$ and $TC$ as above.
	Then the decompositions \eqref{eq:DecompFunctions}, \eqref{eq.DecompVectorFields} and \eqref{eq:DecompPullBackBundle}
	yield the desired isomorphisms.
\end{proof}

\begin{lemma}
	\label{lem:deltaSplitting}
	For every $\psi \in \bigl(\RedSym\ConVectFields(\mathcal{M})\bigr)_\TotalNotWobs$
	it holds
	$
		\D\psi
		\in
		\bigl(\Tensor^2\RedSym\ConVectFields(\mathcal{M})\bigr)_\TotalNotWobs
		\oplus
		\bigl(\Tensor^2\RedSym\ConVectFields(\mathcal{M})\bigr)_\NullNotVan
	$.
\end{lemma}

\begin{proof}
	Let $\psi = X_1 \vee \dots \vee X_k \in \bigl(\RedSym^k\ConVectFields(\mathcal{M})\bigr)_\TotalNotWobs$ be given.
	Then by definition of $D$ we know
	\begin{equation*}
		\D\psi = \D(X_1 \vee \dots \vee X_k)
	= - \sum_{\ell=1}^{k-1} \sum_{\sigma \in \Shuffle(\ell,k-\ell)}
	(X_{\sigma(1)} \vee \cdots \vee X_{\sigma(\ell)})
	\tensor (X_{\sigma(\ell+1)}  \vee \cdots \vee X_{\sigma(k)}).
	\end{equation*}
	By definition of $\bigl(\RedSym\ConVectFields(\mathcal{M})\bigr)_\TotalNotWobs$ we can assume that $X_1 \in \Secinfty(TC^\perp)$,
	and hence we have 
	\begin{equation*}
		(X_{\sigma(1)} \vee \cdots \vee X_{\sigma(\ell)})
		\in \bigl(\RedSym\ConVectFields(\mathcal{M})\bigr)_\TotalNotWobs 
	\qquad\text{or}\qquad
		(X_{\sigma(\ell+1)}  \vee \cdots \vee X_{\sigma(k)})
		\in \bigl(\RedSym\ConVectFields(\mathcal{M})\bigr)_\TotalNotWobs
	\end{equation*}
	for every $\ell$ and every $\sigma$.
	Now a straightforward comparison with the definitions of
	$\bigl(\Tensor^2\RedSym\ConVectFields(\mathcal{M})\bigr)_\TotalNotWobs$
	and
	$\bigl(\Tensor^2\RedSym\ConVectFields(\mathcal{M})\bigr)_\NullNotVan$
	shows the result.
\end{proof}

With this we can now compute the full second constraint Hochschild cohomology:

\begin{theorem}\
	\label{thm:SecondConCohom}
	Let $\mathcal{M} = (M,C,D)$ be a constraint manifold
	and let $\nabla$ be a torsion-free constraint covariant derivative on $\mathcal{M}$.
	Then
	\begin{equation}
		\mathcal{U} \colon \ConVectFields^2(\mathcal{M})_\Wobs
			\oplus \RedSym \Secinfty(D) \vee \Secinfty(TC^\perp)
		\to
		\HHdiff^2(\mathcal{M})_\Wobs
	\end{equation}
	defined by
	\begin{equation}
		\mathcal{U}(X,\psi)
		\coloneqq \hkr(X) + [\Op^\nabla(\D\psi)]
	\end{equation}
	is an isomorphism of vector spaces.
\end{theorem}

\begin{proof}
Let us first show that $\mathcal{U}$ is well-defined.
For this, note that on $\ConVectFields(\mathcal{M})_\Wobs$ it is well-defined by
\autoref{prop:hkrToWobs}.
Moreover, $\D$ maps $\RedSym \Secinfty(D) \vee \Secinfty(TC^\perp)$ to $(\Tensor^2\RedSym\ConVectFields(\mathcal{M}))_\Wobs$
and thus by \autoref{prop:ConMultSymbolCalculus} the map $\mathcal{U}$ is well-defined and as a composition of linear maps it is itself linear.

Next, we show that $\mathcal{U}$ is injective.
Thus assume
$0 = \mathcal{U}(X,\psi) = [X] + [\Op^\nabla(\D\psi)]$.
Since $X$ is antisymmetric and $\Op^\nabla(\D\psi)$ is symmetric, both summands have to vanish separately.
Then by the classical HKR-Theorem it follows $X = 0$.
Moreover, since $\Op^\nabla$ is an isomorphism it follows
that $[\D\psi] = 0$, and hence there exists $\psi' \in (\RedSym\ConVectFields(\mathcal{M}))_\Wobs$ such that 
$\D\psi = \D\psi'$.
Since $\D$ preserves the symmetric degree, we know that
$\psi, \psi' \in \RedSym^k\VectFields(M)$ for some $k \geq 2$
and thus by \autoref{lem:KernelD} from
$\D(\psi - \psi') = 0$ it follows $\psi = \psi' = 0$.

It remains to show that $\mathcal{U}$ is surjective.
Thus let
$\phi \in \CCa^2(\mathcal{M})_\Wobs$
be given such that $\D\phi = 0$.
Then by \eqref{eq:HKRRetract} we know that there exists
$\psi \in \RedSym\VectFields(M)$
such that $\phi = \D\psi + \hkr(X)$,
where $X = \wedge \circ \pr_1^{\tensor 2}(\phi) \in \VectFields^2(M)$.
Since $\phi \in \CCa^2(\mathcal{M})_\Wobs$ we also get
$X \in \ConVectFields^2(\mathcal{M})_\Wobs$.
Then it follows $\D\psi = \phi - \hkr(X) 
\in \CCa^2(\mathcal{M})_\Wobs$.
By \autoref{lem:DecompSymTenAlgebra} we find a decomposition
$\psi = \psi_\Wobs + \psi_\TotalNotWobs$, with
$\psi_\Wobs \in \bigl(\RedSym\ConVectFields(\mathcal{M})\bigr)_\Wobs$
and
$\psi_\TotalNotWobs \in \bigl(\RedSym\ConVectFields(\mathcal{M})\bigr)_\TotalNotWobs$.
Thus we get $\phi = \D\psi_\Wobs + \D\psi_\TotalNotWobs + \hkr(X)$, from which
$\D\psi_\TotalNotWobs \in \CCa^2(\mathcal{M})_\Wobs$
follows.
Since $\ker(\D\at{\RedSym\VectFields(M)}) = \VectFields(M)$
by \autoref{lem:KernelD}
we can assume that
$\psi_\TotalNotWobs = X_1 \vee \cdots \vee X_{k}
\in \bigl(\RedSym^{k}\VectFields(\mathcal{M})\bigr)_\TotalNotWobs$
is a factorizing tensor with $k \geq 2$.
Then by \eqref{eq:SplittingSym} there exists some $i \in \{1,\dotsc,k\}$
with $X_i \in \Secinfty(TC^\perp)$.
Moreover, from \autoref{lem:deltaSplitting} we know that 
$\D \psi_\TotalNotWobs \in \bigl(\Tensor^2\RedSym\VectFields(\mathcal{M})\bigr)_\NullNotVan$.
Then from
\begin{align*}
	\D \psi_\TotalNotWobs
	&=-\sum_{\ell=1}^{k-1}\sum_{\sigma \in \Shuffle(\ell,k-\ell)}
	(X_{\sigma(1)}\cdots X_{\sigma(\ell)})
		\tensor
	(X_{\sigma(\ell+1)} \cdots X_{\sigma(k)})
\end{align*}
and the symmetric grading we know that
\begin{equation*}
	\sum_{\sigma \in \Shuffle(\ell,k-\ell)}
	(X_{\sigma(1)} \vee \cdots \vee X_{\sigma(\ell)})
		\tensor
	(X_{\sigma(\ell+1)} \vee \cdots \vee X_{\sigma(k)})
		\in
	\bigl(\Tensor^2\RedSym\VectFields(\mathcal{M})\bigr)_\NullNotVan
\end{equation*}
for all $\ell = 1, \dotsc,k-1$.
In particular, we get for $\ell=1$
\begin{equation*}
	\sum_{i=1}^k
	X_i \tensor (X_1 \vee \cdots \overset{i}{\wedge} \cdots \vee X_{k})
		\in
	\bigl(\Tensor^2\RedSym\VectFields(\mathcal{M})\bigr)_\NullNotVan.
\end{equation*}
By the definition of 
$\bigl(\Tensor^2\RedSym\VectFields(\mathcal{M})\bigr)_\NullNotVan$
we know that
\begin{align*}
	\sum_{\substack{i\in\{1,\dotsc,k\}\\ X_i \in \VectFields(\mathcal{M})_\TotalNotWobs}}
	X_i \tensor (X_1 \vee \cdots \overset{i}{\wedge} \cdots \vee X_{k})
	&\in \Secinfty(TC^\perp)
		\tensor \bigl(\RedSym\VectFields(\mathcal{M})\bigr)_\NullNotVan
	\\
	\shortintertext{and}
	\sum_{\substack{i\in\{1,\dotsc,k\}\\ v_i \in \VectFields(\mathcal{M})_\Wobs}}
	X_i \tensor (X_1 \vee \cdots \overset{i}{\wedge} \cdots \vee X_{k})
	&\in \Secinfty(D)
	\tensor \bigl(\RedSym\VectFields(\mathcal{M})\bigr)_\TotalNotWobs.
\end{align*}
The first sum collapses to a single summand, since otherwise
$(X_1 \vee \cdots \overset{i}{\wedge} \cdots \vee X_{n})$
could not be an element of
$\bigl(\RedSym_{\tensor}\ConVectFields(\mathcal{M})\bigr)_\NullNotVan$.
Thus there is exactly one $i \in \{1,\dotsc, k\}$ such that $X_i \in \Secinfty(TC^\perp)$.
The second sum shows that for all other $i$ we have $X_i \in \Secinfty(D)$. 
Thus we obtain
\begin{equation*}
	\psi_\TotalNotWobs \in \Sym^{n-1}\Secinfty(D) \vee \Secinfty(TC^\perp)
\end{equation*}
for $n \geq 2$.
Since $\D\pr_1(\psi_\TotalNotWobs) = 0$
we obtain
$\phi = \D\psi_\Wobs +
\mathcal{U}(X,\psi_\TotalNotWobs - \pr_1(\psi_\TotalNotWobs))$,
showing that $\mathcal{U}$ is surjective in cohomology.
\end{proof}

\begin{remark}
	In the definition of $\mathcal{U}$ we use implicitly the choice of a tubular neighbourhood
	of $C$ and a corresponding bump function in order to use the direct sum decomposition of
	\autoref{lem:DecompSymTenAlgebra}.
	In particular, this explains why $\mathcal{U}$ is not a $\Cinfty(M)$-module morphism.
\end{remark}

It should be stressed that by the classical HKR Theorem the Hochschild cohomology $\HCdiff^\bullet(M)$
for a manifold $M$ is given by antisymmetric bi-differential operators of order $(1,1)$.
This is not true in the constraint setting: While $\ConVectFields(\mathcal{M})_\Wobs$ still consists
of (special) antisymmetric bi-differential operators, the contributions of
$\RedSym \Secinfty(D) \vee \Secinfty(TC^\perp)$
yield \emph{symmetric} bi-differential operators.
Moreover, these can be of arbitrarily high differentiation order in each slot.

\begin{proposition}
	Let $\mathcal{M} = (M,C,D)$ be a constraint manifold
	and let $\nabla$ be a torsion-free constraint covariant derivative on $\mathcal{M}$.
	Then the isomorphism
	$\mathcal{U}$ from \autoref{thm:SecondConCohom}
	restricts to an isomorphism
	\begin{equation}
		\mathcal{U} \colon \ConVectFields^2(\mathcal{M})_\Null
			\oplus \RedSym \Secinfty(D) \vee \Secinfty(TC^\perp)
		\to
		\HHdiff^2(\mathcal{M})_\Null.
	\end{equation}
\end{proposition}

\begin{proof}
	The statement follows directly by applying the arguments from the proof of \autoref{thm:SecondConCohom} to the subspace
	$\HHdiff^2(\mathcal{M})_\Null$.
\end{proof}

The canonical morphism $\HHdiff^2(\mathcal{M})_\Wobs \to \HHdiff^2(M)$
corresponds under $\mathcal{U}$ and $\hkr$ to the projection
$(X,\psi) \mapsto X$.
Similarly, the reduction $\HHdiff^2(\mathcal{M})_\Wobs \to \HHdiff^2(\mathcal{M})_\Wobs / \HHdiff^2(\mathcal{M})_\Null \simeq \HHdiff^2(\mathcal{M}_\red)$
corresponds to $(X,\psi) \mapsto [X] \in \VectFields^2(\mathcal{M})_\red \simeq \VectFields^2(\mathcal{M}_\red)$.

\subsection{Reduction of Infinitesimal Star Products}

In deformation quantization a (formal) star product on a manifold $M$ is given by a bilinear map $\star \colon \Cinfty(M)\formal{\hbar} \tensor \Cinfty(M)\formal{\hbar} \to \Cinfty(M)\formal{\hbar}$
with
\begin{equation}
	\star = \mu_0 + \sum_{r=1}^\infty \hbar^r C_r
\end{equation}
where $\mu_0$ is the pointwise multiplication on $\Cinfty(M)$,
fulfilling the following properties:
\begin{cptitem}
	\item $\star$ is associative.
	\item Every $C_r$ is a bi-differential operator on $\Cinfty(M)$.
	\item It holds $f \star 1 = f = 1 \star f$ for all $f \in \Cinfty(M)\formal{\hbar}$.
\end{cptitem}
Note that every such star product induces a Poisson structure on $M$ with Poisson bracket given by$\{\argument,\argument\} = -\I C_1^-$, where $C_1^-(f,g) \coloneqq C_1(f,g) - C_1(g,f)$.
Two star products $\star$ and $\star'$ are called equivalent if there exists
$S = \id + \sum_{r=1}^\infty \hbar^r S_r$
with $S_r \in \Diffop(M)$
such that
\begin{equation}
	\label{eq:EquivStarProducts}
	f \star' g = S^{-1}\bigl(S(f) \star S(g)\bigr).
\end{equation}
It is well-known that the set of equivalence classes of infinitesimal star products, i.e. star products that satisfy associativity only up to order $1$ in $\hbar$, 
is given by the second Hochschild cohomology $\HHdiff^2(M)$.

Let again $\mathcal{M} = (M,C,D)$ be a constraint manifold.
Then a star product $\star$ on $M$ is called \emph{constraint}
if all $C_r$ are constraint bi-differential operators.
Similarly, an equivalence $S$ between two constraint star products is called \emph{constraint} if all $S_r$ are constraint differential operators. 

\begin{proposition}
	\label{prop:InfinitesimalDeformations}
	Let $\mathcal{M} = (M,C,D)$ be a constraint manifold.
	Moreover, let $\star$ and $\star'$ be two constraint star products that agree up to order $k$.
	\begin{propositionlist}
		\item Then $\delta(C_{k+1} - C'_{k+1}) = 0$.
		\item The star products $\star$ and $\star'$ are constraint equivalent up to order $k+1$ if and only if $C_{k+1} - C'_{k+1}$ is constraint exact, i.e.
		exact in $\HCdiff(\mathcal{M})_\Wobs$.
	\end{propositionlist}
\end{proposition}

\begin{proof}
	The first part follows by subtracting the associativity conditions for $\star$ and $\star'$ in order $k+1$ of $\hbar$.
	For the second part assume that $\star$ and $\star'$ are constraint equivalent up to order $k+1$.
	Since $\star$ and $\star'$ agree up to order $k$ we can assume that the equivalence is given by 
	$S = \id + \hbar^{k+1}S_{k+1} + \hbar^{k+2}(\dotsc)$
	with $S_{k+1} \in \HCdiff^1(\mathcal{M})_\Wobs$.
	Then \eqref{eq:EquivStarProducts} yields in order $k+1$ of $\hbar$ exactly
	$\delta S_{k+1} = C_{k+1} - C'_{k+1}$.
	Finally, assume that $C_{k+1} - C'_{k+1} = \delta B$ for some 
	$B \in \HCdiff^1(\mathcal{M})_\Wobs$.
	Then $S \coloneqq \id + \hbar^{k+1} B$
	is a constraint equivalence between $\star$ and $\star'$ up to order $k+1$.
\end{proof}

\begin{corollary}
	Let $\mathcal{M} = (M,C,D)$ be a constraint manifold.
	The set of constraint equivalence classes of infinitesimal constraint star products on $\mathcal{M}$ is given by
	$\ConVectFields^2(\mathcal{M})_\Wobs
	\oplus \RedSym \Secinfty(D) \vee \Secinfty(TC^\perp)$.
\end{corollary}

\begin{proof}
	Apply \autoref{prop:InfinitesimalDeformations} for $k=0$ and use
	the identification
	$\HHdiff^2(\mathcal{M})_\Wobs
	\simeq \ConVectFields^2(\mathcal{M})_\Wobs
	\oplus \RedSym \Secinfty(D) \vee \Secinfty(TC^\perp)$
	from \autoref{thm:SecondConCohom}.
\end{proof}

We have seen before how $\HHdiff^2(\mathcal{M})_\Wobs$ is related to $\HHdiff^2(M)$
and $\HHdiff^2(\mathcal{M})_\red$.
We now understand that elements in $\RedSym \Secinfty(D) \vee \Secinfty(TC^\perp)$
correspond to non-equivalent infinitesimal constraint star products which become equivalent
when considered as star products without compatibility with reduction.
Moreover, all these infinitesimal constraint star products become equivalent after reduction.

{
	\footnotesize
	\printbibliography[heading=bibintoc]
}

\end{document}